\documentclass[11pt,reqno]{amsart}
\usepackage{amsfonts,amsmath,amssymb}
\usepackage[all]{xy}
\newtheorem{thm}{Theorem}[section]
\newtheorem{proposition}[thm]{Proposition}
\newtheorem{corollary}[thm]{Corollary}
\newtheorem{lemma}[thm]{Lemma}
\newtheorem{definition}[thm]{Definition}
\newtheorem{example}[thm]{Example}
\newtheorem{remark}[thm]{Remark}

\usepackage{graphicx}
\usepackage{epstopdf}

\title [Combinatorics of  irreducible characters  for Lie superalgebra $\frak{gl}(m,n)$]{ Combinatorics of  irreducible characters  for Lie superalgebra $\frak{gl}(m,n)$}

\begin{document}
\author{A.N. Sergeev}\address{Department of Mathematics, Saratov State University, Astrakhanskaya 83, Saratov 410012   Russian Federation.}
 \email{SergeevAN@info.sgu.ru}
\maketitle

\begin{abstract} In this paper we give a new formula for characters of finite dimensional irreducible  $\frak{gl}(m,n)$ modules. We use two main ingredients: Su-Zhang formula and Brion's theorem.
 \end{abstract}
 
\tableofcontents

\section{Introduction} In this paper we prove a new formula for characters of irreducible finite dimensional $\frak{gl}(m,n)$ modules. Our formula is the same type as Su-Zhang formula (\cite{YZ}). Namely Su and Zhang proved  that for irreducible  module $L(\lambda)$ the following equality holds true
$$
ch\,L(\lambda)=\sum_{\mu\in \mathcal {C}^{\text{Trun}}_{\lambda}}(-1)^{|\lambda-\mu|}ch\,K(\mu)
$$
where $|\lambda-\mu|$  is some integer number, $\mathcal{C}^{\text{Trun}}_{\lambda}$ is the subset in $\Bbb Z^{m+n}$ and $K(\mu)$ is Kac module. Using this formula they proved a Weyl type character formula.  First they represent  $\mathcal{C}_{\lambda}^{\text{Trun}}$ as the union of fundamental domains under the symmetric group. Then for every fundamental domain they (implicitly) apply Brion's theorem. 

In our approach  instead of using weights we use the language of weight diagrams \cite{BS}. Then we apply Brion's theorem to the  polyhedron for which $\mathcal{C}^{\text{Trun}}_{\lambda}$ is the set of integer points. The Su-Zhang formula contains up to $r!2^r$ summands where $r$ is the degree of atypicality. Our  formula always contains $2^{r-s}$ summands  where s is the number of connected components of the weight diagram. 

Our  main result can be formulated in the following way (see for details  section 5).
\begin{thm}  Let $L(\chi)$ be a finite dimensional irreducible module over Lie superalgebra $\frak{gl}(m,n)$ with the highest  weight $\chi$. Then the following equality holds true

$$
D\,chL(\chi)=\sum_{w\in W_0}\varepsilon(w)w\left(\frac{e^{\chi+\rho}\theta(\Gamma_f,-e^{\alpha_1},\dots,-e^{\alpha_r})}{\prod_{\alpha\in S_{\chi}}(1+e^{-\alpha})}\right)
$$
where $\Gamma_{f}$ is the graph  which is explicitly constructed  from  the weight diagram $f$  of $\chi$, $S_{\chi}=\{\alpha_1<\dots<\alpha_r\}$  is a maximal orthogonal  set of atypical roots and $\theta(\Gamma_f,t_1,\dots,t_r)$ is a Laurent polynomial given by explicit formula. 
\end{thm}
We also give one  more formula (see the end of section 5) of the same type with  the  less number of summands. We were motivated mainly by papers \cite{BFM}, \cite{M}, \cite{YZ}.

\section{Preliminaries}
Let us remind that the Lie superalgebra $\frak{gl}(m,n)$ is the Lie  superalgebra of the  linear transformations of a $\Bbb Z_2$ graded vector space $V=V_0\oplus V_1$ ($V$ is also called the standard representation of $\frak g$). We have 
$$
\frak g_{\bar0}=\frak{gl}(m)\oplus\frak{gl}(n),\quad\frak g_{\bar1}=V_0\otimes V_1^*\oplus V_1\otimes V_0^*.
$$ 
We also have $\Bbb Z$ graded decomposition 
$
\frak g=\frak g_{-1}\oplus \frak g_0\oplus \frak g_1
$
where
$$
\frak g_{-1}=V_1\otimes V^*_0,\,\, \frak g_{1}=V_0\otimes V^*_1\,.
$$
Let us fix  bases in $V_0=<e_1,\dots,e_m>$ and $V_1=<f_1,\dots,f_n>$ respectively. 
Let  $\frak{b}$ be the subalgebra of upper triangular matrix in $\frak{gl}(m,n)$ and $\frak{k}$ be  the    subalgebra of diagonal matrix in  $\frak{gl}(m,n)$ in the above basis. By   $\varepsilon_1,\dots,\varepsilon_m,\delta_1,\dots,\delta_n$  we will denote  the weights of standard representation with respect  to $\frak k$. The corresponding system of positive roots  $R^+=R^+_0\cup R^+_1$  of $\frak{gl}(m,n)$ can  be described in the following way
$$
R^+_{0}=\{\varepsilon_i-\varepsilon_j\,:\, 1\le i< j\le m\,:\, \delta_k-\delta_l,\, 1\le k<l\le n \}
$$
$$
R_1^+=\{\varepsilon_i-\delta_k,\, 1\le i\le m,\,1\le k\le n\}.
$$

 Let also  
$$
P=\{\chi=\lambda_1\varepsilon_1+\dots+\lambda_m\varepsilon_m+\mu_1\delta_1+\dots+\mu_n\delta_n,\mid n_i,m_j\in \Bbb Z\}
$$
be the weight lattice and 
$$
P^+=\{\chi\in P\mid \lambda_i-\lambda_j\ge0,\,i<j\,:\mu_k-\mu_l\ge0,\,k<l\}
$$
be the set of highest weights.

We will use the following parity on the weight lattice due to C. Gruson  and V. Serganova \cite{GS1}  and Brundan and Stroppel \cite{BS} by saying  that $\varepsilon_i$ (resp. $\delta_j$) is even (resp. odd). It is easy to check that every finite dimensional module $L$ can be represented in the form
$$
L=L^+\oplus L^{-}
$$
where $L^+$ is the submodule of $L$ in which weight space has the same parity as the corresponding weight and $L^{-}$ is the submodule in which  the parities differ.  We should note that this construction is a particular case of Deligne construction category $Rep(G,z)$ from the paper \cite{D}  for $G=GL(m,n)$ and $z=diag(\underbrace{1,\dots,1}_{m},\underbrace{-1,\dots,-1}_{n}).$

Let us denote by $\mathcal F$ the category of finite dimensional modules over $\frak{gl}(n,m)$ such that every module in $\mathcal F$ is semisimple over Cartan  subalgebra $\frak k$ and and all  its weights are in $P$.
By  $K(\mathcal F )$ we will denote the quotient of the Grothendieck ring of   $\mathcal F$ by the relation $[L]-[\Pi(L)]=0$ where $\Pi(L)$ is the module  with the shifted parity $\Pi(L)_0=L_1, \Pi(L)_1=L_0$ and $x*v=(-1)^{p(x)}xv,x\in\mathfrak{gl}(m,n).$ For every $L\in \mathcal F$ we can define 
$$
ch\,L=\sum_{\chi}\dim L_{\chi}e^{\chi}
$$
where the sum is taken over all weights of $L$. It is easy to see that $ch\,L$ is well defined function on $K(\mathcal F)$.

The ring $K(\mathcal F)$ can be describe explicitly in the following way.  Let 
$$
P_{m,n}=\Bbb Z[x_1^{\pm1},\dots, x_m^{\pm1},\, y_1^{\pm1},\dots, y_n^{\pm1}]
$$
be the ring of Laurent polynomials  in variables $x_1,\dots,x_m$ and $y_1,\dots, y_n.$ 

If we set $x_i=e^{\varepsilon_i},\, y_j=e^{\delta_j}$ then we get a character map 
$$
ch : K(\mathcal F)\longrightarrow P_{m,n}.
$$
 Let   also 
$$
\Lambda^{\pm}_{m,n}=\{f\in P_{m,n}^{S_m\times S_n}\mid x_i\frac{\partial f}{\partial x_i}+y_j\frac{\partial f}{\partial y_j}\in(x_i+y_j) \}
$$
be the subring of $P_{m.n}$ of supersymmetric Laurent  polynomials. 
\begin{thm}\cite{SV1} The   ring  $K(\mathcal F)$ is isomorphic to the ring $\Lambda^{\pm}_{m,n}$ under the character map.
\end{thm}
\begin{remark} Actually in the paper \cite{SV1} a slightly different versions of Grothen-dieck ring  and  the algebra  $\Lambda^{\pm}_{m,n}$  were considered. But it is easy to check that  they are isomorphic to our ones. We prefer to use characters  instead of supercharacters in this paper in order to avoid some unnecessary  signs.
\end{remark}
It will be needed later an explicit description of  the projective covers of the  irreducible finite dimensional modules due to Brundan \cite{Brun}. We give the description here in a slightly different way. 

First let us for  any $\chi\in P^+$ define a pair of sets 
$$
A=\{(\chi+\rho,\varepsilon_1),\dots, (\chi+\rho,\varepsilon_m)\},\quad B=\{(\chi+\rho,\delta_1),\dots, (\chi+\rho,\delta_n)\}
$$
where 
$$
\rho=\frac12\sum_{\alpha\in R_0^+}\alpha-\frac12\sum_{\alpha\in R_1^+}\alpha+\frac12(n-m+1)(\sum_{i=1}^m\varepsilon_i-\sum_{j=1}^n\delta_j)
$$
$$
=\sum_{i=1}^m(1-i)\varepsilon_i+\sum_{j=1}^n(m-j)\delta_j.
$$
Our $\rho$ is slightly different from the standard one but it is more convenient since the elements of $A$ and $B$ are integers.
So instead of highest weights we will use the set of pairs $(A,B)$ such that  $A,B\subset\Bbb Z$ and $|A|=m,\,|B|=n$.  We will  also use the language of diagrams  which is due to Brundun and Stroppel \cite{BS} but we will use it  here  in a form due  to I. Musson and V. Serganova  \cite{MS}.

\begin{definition} Let $(A,B)$ be a pair of subsets in $\Bbb Z$ such that $|A|=m,\, |B|=n$. Then the  corresponding weight  diagram is the following function on $\Bbb Z$
$$
f(x)=\begin{cases} \times,\,\,x\in A\cap B\\
\circ,\,\,x\notin A\cup B\\
>,\,\,x\in A\setminus B\\
<,\,\, x\in B\setminus A
\end{cases}.
$$
Geometrically we can picture a diagram  $f$ as $\Bbb Z$-line with  $f(x)$ above $x$.
\end{definition}

   In the category $K(\mathcal F)$ we have three important classes of modules: irreducible modules $\{L(f)\}$; Kac modules $\{K(f)\}$; projective indecomposable modules $\{P(f)\}$. Characters of Kac modules can be easily described.
\begin{definition} Let $\chi\in P$ then we define the alternation operation $J$  by the formula
$$
J(e^{\chi})=\sum_{\sigma\in S_m\times S_n}\varepsilon(\sigma)e^{\sigma(\chi)}
$$ 
\end{definition}
 Let 
$$
D=\frac{\prod_{\alpha\in R_0^+}\left(e^{\alpha/2}-e^{-\alpha/2}\right)}{\prod_{\alpha\in R_1^+}(e^{\alpha/2}+e^{-\alpha/2})}
$$
and $f$ be a diagram such that 
$$
f^{-1}(\times,>)=\{a_1>a_2>\dots>a_m\},\quad f^{-1}(\times,<)=\{b_1<b_2<\dots<b_n\}
$$
then we  have 
$$
D ch\, K(f)=J(e^{\omega(f)}),\quad \omega(f)=\sum_{i=1}^ma_i\varepsilon_i-\sum_{j=1}^nb_j\delta_j=\chi+\rho
$$
\begin{definition} Let $f$ be a weight diagram.  The corresponding cap diagram can be obtained in the following way. Take the rightmost  $\times$ and make the cap by joining it to the first $\circ$ on the right. Then take the next $\times $ to the left  and make the cap by joining it  to the first $\circ$ on the right which is not the end of a cap, and so on. 

We can also define the cap diagram in another way. Let $f(a)=\times$. Let us take the first $c$ on the right such that $f(c)=\circ$ and  the numbers of $\times$-s and $\circ$-s in the interval $(a,c)$  are the same. Then draw the cap joining $a$ and $c$.
\end{definition}
 Let  $f(a)=\times$ then the right end  of the corresponding cap will be denoted by $\tilde a$ and by $\sigma_a$ we will denote the transposition permuting $a$ and $\tilde a$ 
 
  \begin{definition} Let us denote by $\mathcal P(f)$ the set $\sigma_C(f)$ where $C\subset f^{-1}(\times)$ and $\sigma_C=\prod_{c\in C}\sigma_c$.   
\end{definition}

 \begin{thm} (Brundan \cite{Brun}) Let $P(f)$  be the projective cover of an irreducible module $L(f)$ then in the character ring we have 
 $$
 ch\,P(f)=\sum_{g\in\mathcal P(f)}  ch\,K(g)
 $$
 \end{thm} 

We need a topology on the character ring $K(\mathcal F)$. Let  us define 
 $$
 n(f)=\sum_{f(a)\ne\circ,\,>}a
 $$ 
 and $K(\mathcal F)_d,\,d\in \Bbb Z$ be the subgroup generated by the irreducible modules $L(f)$ such that  $n(f)\le d$. Then it is easy to see that $\{K(\mathcal F)_d,\,d\in \Bbb Z\}$ is an ascending  filtration on the group $K(\mathcal F)$. So we can consider the infinite series consisting of Kac modules. It is easy to see that every irreducible module can be represented  in a unique way as the sum of Kac modules:
 $$
 L(f)=\sum_{g}a_{gf}K(g)
 $$ 
 Let $f$ be a diagram and $\varphi : f^{-1}(\times\,,<,\,>)\rightarrow \Bbb Z$ be an injection such that $\varphi(a)=a$ for any $a\in f^{-1}(<,\,>)$.  Then we can define the following new diagram $g$
 $$
 g^{-1}(>)=f^{-1}(>),\, \,\,\,\,  g^{-1}(<)=f^{-1}(<),\,\,\,\,\, g^{-1}(\times)=\varphi(f^{-1}(\times)).
  $$
  We will denote the above diagram by $\varphi(f)$.
  We can also define a sign of $\varphi$ by the following formula
  $$
  \varepsilon(\varphi)=\sum_{a\in f^{-1}(\times)}(a-\varphi(a)-n(a,b)),\quad n(a,b)=|(a,\varphi(a))\cap\{f^{-1}(<,>\}|
  $$
  We should note, that $(a-\varphi(a)-n(a,b))$ is the number of $\circ$-s inside the interval $(a, \varphi(a))$.
  
  \begin{definition} Diagram $f$ is called core free if $f^{-1}(<)=f^{-1}(>)=\emptyset$.  Let $f$ be any diagram. We will denote by $f^{\sharp}$ the diagram obtaining from $f$ by deleting all symbols $<,\,>$. If $\varphi$ is as above then we have the corresponding map 
  $$
  \varphi^{\sharp}:(f^{\sharp})^{-1}(\times)\rightarrow  (g^{\sharp})^{-1}(\times)
  $$.
  \end{definition}

 \begin{lemma}\label{core}  The following equality holds true $\varepsilon(\varphi)=\varepsilon(\varphi^{\sharp})$
 \end{lemma}
 \begin{proof} Let $f(a)=\times$ and $a^{\sharp}$ be the corresponding number for $f^{\sharp}$ then we have 
 $
 a^{\sharp}=a-n(a)
 $
 where $n(a)$ is the number of symbols $<,\,>$ on left on $a$. So we have 
 $$
 \varepsilon(\varphi^{\sharp})=\sum_{a^{\sharp}\in (f^{\sharp})^{-1}(\times)}(a^{\sharp}-\varphi^{\sharp}(a^{\sharp}))=\sum_{a\in f^{-1}(\times)}(a-n(a)-\varphi(a)+n(\varphi(a))=\varepsilon(\varphi)
 $$
 \end{proof}
 \section{Su-Zhang formula} 
 
 In this section we give a different proof   of Su-Zhang formula for the decomposition irreducible module in terms of Kac modules. We also  formulate their result in a  different terms.
 
 \begin{definition}\label{order} Let $f$ be a diagram.   Let us introduce a partial order on the set $f^{-1}(\times)$ by using its cap diagram. Let $a,b\in f^{-1}(\times)$ and $C_a, C_b$ are the corresponding caps. We say that $a\dashv b$ if  $ C_b$ is located  under  $C_a$. It is easy to see that  this is indeed a partial order on the set $f^{-1}(\times)$.
\end{definition}
\begin{definition}\label{dw}
Let us  denote by $W(f,\Bbb Z)$ the set of $\varphi: f^{-1}(\times)\rightarrow \Bbb Z$  such that:

$1)$ $\varphi$ is injection

$2)$ $\varphi$ is  a morphism  of  the ordered sets, where $f^{-1}(\times)$ is ordered by means of $\dashv$  and $\Bbb Z$ is ordered in standard way

$3)$ $\varphi(a)\le a$ for any $a\in f^{-1}(\times)$

$4)$ $\varphi(a)=a$ if  $a\in f^{-1}(<,\,>)$.

We also denote by $W(f, g)$ the subset of  $\varphi\in W(f,\Bbb Z)$  such that $\varphi(f)=g$.
\end{definition}
\begin{thm} \label {t1}The following equality holds true
$$
ch\,L(f)]=\sum_{\varphi\in W(f,\Bbb Z)}(-1)^{\varepsilon(\varphi)}ch\,K(\varphi(f))
$$
\end{thm}

Let us fix two nonnegative integers $m,n$. We will denote by $F(m,n)$ the set of diagrams $f$  such that
$$
|f^{-1}(>)|+|f^{-1}(\times)|=m,\,\,  |f^{-1}(<)|+|f^{-1}(\times)|=n
$$
Now let $\mathcal A,\mathcal B$ be two matrixes enumerated by elements of $F(m,n)$  and defined by the following rules
$$
a_{f,g}=\delta_{\{g\in \mathcal {P}(f)\}},\quad b_{f,g}=\sum_{\varphi(f)=g}(-1)^{\varepsilon(\varphi)}
$$
The statement of the Theorem is equivalent the following equality
$$
(\mathcal A\mathcal B)_{fg}=\sum_{h}a_{f,h}b_{h,g}=\sum_{h,\,\varphi(g)=h\in \mathcal P(f)}(-1)^{\varepsilon(\varphi)}=
$$
$$
\sum_{\varphi(g)\in\mathcal P(f)}(-1)^{\varepsilon(\varphi)}=(P(f),L(g))=\delta_{f,g}
$$
where  $(P(f),L(g))$ means the canonical bilinear form.

It not difficult to see by using Lemma \ref{core}  that we only need to prove core free case.  In this case  instead of diagram $f$ we will use the corresponding set $f^{-1}(\times)$ and for $a\in f^{-1}(\times)$ we will denote by $\tilde a$ the right end of the corresponding cap. Let us  introduce the following set
 $$
 W(A,\mathcal P(B))=\{f\in W(A,\Bbb Z)\mid  f(A)\in\mathcal P(B)\}.
 $$

 \begin{definition} Let $A\subset\Bbb Z$. We will denote by $i(A)$ the following  subset  
$$
i(A)=\{a\in A\mid \tilde a-a=1\}
$$
\end{definition}
It is clear that for any $A$ the set $i(A)$ is not empty. It is also easy to see that $A$ is totally disconnected if and only if $i(A)=A$ and $A$ is totally connected if and only if $|i(A)|=1$.

\begin{definition}  Let $A\subset\Bbb Z,$ and $a\in A$.  Then by $C_a$ we will denote  the upper half of a circle joining two integers $a$  and $\tilde a$ and call it a cap. The  set $\{C_a\mid a\in A\}$ we will call a cap diagram.
\end{definition}

Now we are ready to formulate a theorem that gives an explicit expression for  $L(A)$.

\begin{thm}\label{w} The following equality holds true
$$
ch\,L(A)=\sum_{f}(-1)^{\varepsilon(f)}ch\,K(f(A))
$$
where sum is taken over all $f\in W(A,\Bbb Z)$  and  $\varepsilon(f)=\sum_{x\in A}(x+f(x))$.

\end{thm}

Before proving the Theorem let  us prove several technical Lemmas. 

\begin{lemma}\label{1} Let $b\in i(B)$ and $C\in \mathcal P(B)$. If $f:A\rightarrow C$ is a bijection and morphism of ordered sets (as above) then $\tau_a\circ f$ is  a bijection and  a morphism of ordered sets too.
\end{lemma}
\begin{proof} Clearly $\tau_b\circ f$ is a bijection  $A$ on $\tau_b(C)$. Besides  $\tau_b : C\rightarrow \tau_b(C)$ is a morphism of ordered sets in the standard sense since  $C$ and $\tau_b(C)$ contain only one of the elements $b,b+1$.Therefore  $\tau_b\circ f :A\rightarrow \tau_b(C)$ is a bijection and  a morphism of ordered sets too.
\end{proof}

\begin{lemma} \label{2}
 If $b\in i(B)$ and $b\notin A$ then  $W(A,\mathcal P(B))$ is invariant with respect to $\tau_b$.
\end{lemma}
\begin{proof} Let $f\in W(A,\mathcal P(B))$ then    by Lemma \ref{1} $\tau_b\circ f$ is a morphism of ordered sets. Let us prove that $(\tau_b\circ f)(a)\le a$. If $f(a)\ne b,b+1$ then  $(\tau_b\circ f)(a)=f(a)\le a$. If $f(a)=b+1$ then  $(\tau_b\circ f)(a)=b<b+1\le a$.  If $f(a)=b$ then $b\le a$ besides $b\notin A$. Therefore  $b<a$ and  $(\tau_b\circ f)(a)=b+1\le a$.   
\end{proof}
\begin{lemma}\label{3} Let $b\in i(B)$ and $b\in A$. Then the following set
$$
W_{<b}(A,\mathcal P(B))=\{f\in W(A,\mathcal P(B))\mid f(b)<b\}
$$
is invariant with respect to $\tau_b$.
\end{lemma}
\begin{proof} Let $f\in W_{<b}(A,\mathcal P(B))$ then   $(\tau_b\circ f)(b)=f(b)<b$. Besides   by Lemma \ref{1} $\tau_b\circ f$ is a morphism of  ordered sets.  Let us prove, that $\tau_b(f(a))\le a$ for any $a\in A$. If $f(a)\ne b,b+1$ then  $(\tau_b\circ f)(a)=f(a)\le a.$  If $f(a)=b+1$ then $\tau_b(f(a))=b<b+1\le a$. If $f(a)=b$ then $b\le a$ and $b\ne a$ since $f(b)<b$. Therefore  $b<a$ and $\tau_b(f(a))=b+1\le a$.
\end{proof}

Let us introduce the following set 
$$
W_{b}(A,\mathcal P(B))=\{f\in W(A,\mathcal P(B)\mid f(b)=b\}
$$
Then we have 
$$
W(A,\mathcal P(B))=W_{b}(A,\mathcal P(B))\cup W_{<b}(A,\mathcal P(B))
$$
\begin{proposition}\label{4} Let $b\in i(B),\,b,b+1\in A$. Then $W_b(A,\mathcal P(B))=\emptyset$.
\end{proposition}
\begin{proof} Suppose that $f\in W_b(A,\mathcal P(B))$ and $f: A\rightarrow C$. Since $b,b+1\in A$ then $b\dashv b+1$. Therefore $b=f(b)<f(b+1)\le b+1$. But $f(b+1)\ne b+1$. This is a contradiction. Proposition is proved.
\end{proof}

\begin{definition}
\end{definition}
Let us define the following function   $\theta : \Bbb Z\setminus \{b,b+1\}\rightarrow\Bbb Z$
$$
\theta(x)=\begin{cases}x,\,\,x<b \\
x-2,\,\, x>b+1
\end{cases}
$$
It is easy to see that $\theta$ is isomorphism of ordered set.

\begin{lemma}\label{5} Suppose, that $A,C\subset\Bbb Z $ and $b\in A\cap C,\,\,\,\, b+1\notin A\cup C$. Then the correspondence $f\rightarrow \theta\circ f\circ \theta^{-1}$ defines a bijection between 
$
W_b(A,C)$ and $ W(\theta(A\setminus\{b\}), \theta(C\setminus\{b\})).
$
\end{lemma}
\begin{proof} \,

 It is easy to check  that  $\theta: A\setminus\{b\}\rightarrow \theta(A\setminus\{b\})$   is an isomorphism of ordered  set if we consider $ A\setminus\{b\}$ as the ordered  subset of $A$ and the order defined by the set of arks $D_A$.  
 
 Therefore 
if $f\in W_b(A,C)\,$  then$\,\,  \theta\circ f\circ \theta^{-1}\in  W(\theta(A\setminus\{b\}), \theta(C\setminus\{b\}).
$
\\

In order to prove that our map is a bijection we will construct the inverse map. Let $g\in W(\theta(A\setminus\{b\}), \theta(C\setminus\{b\})$ then we define $f:A\rightarrow C$ by the following rule
$$
f(a)=\begin{cases}b,\,\, a=b\\
(\theta^{-1}\circ g\circ \theta)(a),\,\, a\ne b
\end{cases}
$$
So we need to check that $g=\theta\circ f\circ \theta^{-1}$  and that $f\in W_b(A,C)$. The first property is clear.  The second property follows from the fact that $\theta$ is isomorphism of ordered  sets $A\setminus\{b\}$ (as the subset of $A$ with respect to the order defined by $D_A$) and  $\theta(A)$ with respect to the order  defined by $D_{\theta(A\setminus\{b\})}$. Besides $\theta$ is isomorphism of ordered set $\Bbb Z\setminus\{b,b+1\}$ and $\Bbb Z$ with respect to the standard order on integers. The only additional statement we need to prove is the following : if $a\dashv b$ then $f(a)<f(b)=b$. But we have $g(\theta(a))\le \theta(a)$. Therefore $f(a)=\theta^{-1}(g(\theta(a)\le a<b$.
Lemma  is proved.
\end{proof}

\begin{corollary} $(P(B),L(A))=\delta_{A,B}$.
\end{corollary}
\begin{proof} Let us use induction on $|A|$ - the number of elements in $A$. If $|A|=1$ then the statement is clear. Let $|A|>1$.  If $B=A$ then  $(P(A),L(A))=1$. Let $B\ne A$. If $iA)\cap i(B)=\emptyset$ then by Lemmas \ref{2},\ref{3} and Proposition \ref{4} we have $(P(B),L(A))=0.$ If $i(A)\cap i(B)\ne\emptyset$ then  let us take any element $b$ from this intersection  and  let us apply Lemma \ref{5}. Since $A\ne B$ we have $\theta(A\setminus\{b\})\ne \theta(B\setminus\{b\})$. By Lemmas \ref{5},\,\ref{4} we have
$$
(P(B),L(A))=(-1)^b(P(\theta(B\setminus\{b\}),L(A\setminus\{b\}))=0
$$
\end{proof}

So we see that we proved Theorem \ref{w} and therefore Theorem \ref{t1}.
\begin{remark}We can reformulate  Theorem \ref{t1} in geometric terms. Let
$$
\Bbb Z\setminus f^{-1}(\circ)=\{c_1<c_2<\dots<c_N\},\,i=1,\dots,N
$$

\begin{definition} Let $f$ be a diagram. Let us define the subset $D_f\subset \Bbb R^N\,$:  $(x_1,\dots,x_N)\in D_{f}$ if and only if the following conditions are fulfilled

$1.$ $ x_i\le c_i $ if $ f(c_i)=\times$ 

$2$. $x_i=c_i$ if $f(c_i)=\,<,\,>$

$3.$  $x_i\le x_j$ if $c_i\dashv c_j$ 
\end{definition}

Any point  $(x_1,\dots,x_N)\in D_f$ with $ x_i\in \Bbb Z,\,i=1,\dots, N$  can be  considered as the diagram $g$ such that
$
g(c_i)=f(c_i)
$
if $f(c_i)=\,<,\,>$ and $g(c_i)=x_i$ if $f(c_i)=\times$. We will denote  this diagram by $g_x$.

Let us also define a linear map $\pi_f: \Bbb R^N\rightarrow R^{m+n}$ by the following formula
$$
\pi_f(e_k)=\begin{cases}\varepsilon_i,\,\, f(c_k)=\,>,\, c_k=a_i\\
-\delta_j,\,\, f(c_k)=\,<,\, c_k=b_j\\
\varepsilon_i-\delta_j,\,\, f(c_k)=\times,\, c_k=a_i=b_j
\end{cases}
$$
where $e_k,\,k=1,\dots, N$ the standard basis in $\Bbb R^{N}$.
\begin{corollary}\label{t2} The following equality holds true
$$
D ch\,L(f)]=(-1)^{S(f^{-1}(\times))}\sum_{x\in D_{f}}(-1)^{S(g_x^{-1}(\times))}J(e^{\pi_f(x)})
$$
where for $X\subset\Bbb Z$  we define $S(X)=\sum_{x\in X}x$.
\end{corollary}
\begin{proof} We have
$$
ch\,L(f)]=\sum_{x\in D_f}(-1)^{\varphi_x}ch\,K(g_x)
$$
where $\varphi_x(f)=g_x$. Therefore 
$$
Dch\,L(f)=\sum_{x\in D_f}(-1)^{\varphi_x}J(e^{\omega(g_x)})
$$
So we see that we only need to prove that 
$$
J(e^{\omega(g_x)})=(-1)^{\tau}J(e^{\pi_f(x)}),\quad\text{where}\,\,\, \tau=\sum_{c\in f^{-1}(\times)}n(c,\varphi_x(c)).
$$
The above formula follows from the fact that there are $\sigma_1\in S_m,\sigma_2\in S_n$ such that $\sigma_1\sigma_2(\omega(g_x))=\pi_f(x)$ and 
$
sign (\sigma_1\sigma_2)=(-1)^{\tau}
$.
\end{proof}
\end{remark}

\section{Graphs and polyhedra}

\begin{definition}\label{graph}  Let  $f$ be a weight diagram and let  $f^{-1}(\times)=\{c_1,\dots,c_r\}$ be ordered as in Definition \ref{order}. Let us denote by $\Gamma_f$ the directed graph with the set of vertexes $\{1,\dots, r\}$  and $i\rightarrow j$  if $c_i\dashv c_j$ and the interval $(c_i,c_j)$ is empty.
\end{definition}

Let $\Gamma$ be a  directed graph without multiple edges  and cycles (if we ignore the orientation of  $\Gamma$) with $r$ vertexes enumerated by integers  $1,\dots,r$. Let also  $c_1< c_2<\dots<c_r$ be a sequence of integers such that if $i\rightarrow j$ is a directed edge then $c_i< c_j$.  Let us also define $c_{\Gamma(i)}$ to be  equal to  the $\min\{c_k\}$  where $c_k$ run over  connected component containing $c_i$.  We denote by $M(\Gamma)$ the set of subgraphs of $\Gamma$ with the same set of vertexes.

We will denote by $V(\Delta)$ the set of vertexes  of $\Delta$ and by $E(\Delta)$ the set of the edges of $\Delta$. We also will denote by 
$[x_i\le c_i]$ and $[x_i\le x_j]$ the characteristic functions of the corresponding sets.

\begin{definition}
So for  every graph as above we can define a polyhedron  $D_{\Gamma}$ in the space $\Bbb R^r$ by its characteristic function  
$$
[D_{\Gamma}]=\prod_{i\in V(\Gamma)}[x_i\le c_i]\prod_{(i\rightarrow j)\in E(\Gamma)}[x_i\le x_j]
$$ 
\end{definition}
\begin{remark}
We should note the if $\Gamma=\Gamma_f$ then $D_{f}=D_{\Gamma_f}$.
\end{remark}
Our aim in this section is to calculate the generation function of integer point in the polyhedron $D_f$.
 \begin{definition} The generation function of a polyhedron $P\subset \Bbb R^r$  is 
  $$
  G(P)(t_1,\dots, t_n)=\sum_{x\in \Bbb Z^n\cap P}t_1^{x_1}\dots t_n^{x_r}
  $$
  \end{definition}
 We also need  Brion's theorem.
\begin{thm}(\cite{Bar}, section 6)

Let  $D$ be a polyhedron  and $V$ be the set of its vertexes. Then we have
$$
G(D)=\sum_{M\in V} G(tcone(M))
$$
where $G$ is a generating function of integer points and $tcone(M)$ is the tangent cone at the point $M$.
\end{thm}

The next Lemma describes the vertexes of the polyhedron $D_{\Gamma}$.

\begin{lemma}\label{vertex}   Let $\Delta\in M(\Gamma)$.  Then the point $M_{\Delta}=(x_1,\dots,x_r)$ such  that $x_i=c_{\Delta(i)}$ is a vertex of the polyhedron $D_{\Gamma}$. This correspondence is a bijection between  the set $M(\Gamma)$ and the set of vertexes of the polyhedron $D_{\Gamma}$.
\end{lemma}
\begin{proof} We can suppose that $\Gamma$ is a connected graph. Let $\Delta\in M(\Gamma)$ and  $\Delta=\cup\Delta_{\alpha}$ be its decomposition into connected components. Let also $i_\alpha\in\Delta_{\alpha}$ be the minimal vertex. Then $M=(x_1,\dots,x_r)$ has the following characteristic function
$$
[M_{\Delta}]=\prod_{\alpha}[x_{i_{\alpha}}=c_{i_{\alpha}}]\prod_{(i\rightarrow j)\in E(\Delta)}[x_i= x_j].
$$
Since the number of equations is equal to $|V(\Gamma)|$ the point $M$ is a vertex of $D_{\Gamma}$. Now let us prove that  the correspondence  $\Delta\rightarrow M_{\Delta}$ is a bijection.

 Let  $M\in D_{\Gamma}$  be a vertex and $f_1(M)=\dots=f_r(M)=0$ for some $f_i=x_i-c_i$ or $f_i=x_i-x_j$ with linear independent linear parts.  Let us consider the last  equation $f_r=0$. Since this vertex is maximal  there exists not more than one edge  containing this vertex. 
Therefore the sequence $f_1,\dots,f_n$ may contains  one of the equations $x_r-x_i, x_r-c_r$ or two of them. 

In the first case the  sequence $f_1,\dots,f_{r-1}$  does not contain $x_n$. Consider the graph $\tilde G$ which can be  obtained   from $\Gamma$ by deleting the vertex $v_n$  with number $n$. By induction there exists $\tilde\Delta\in M(\tilde\Gamma)$  such that $(x_1,\dots,x_{r-1})= M_{\tilde\Delta}$.
Let $f_n=x_n-c_n$. If the vertex  $v_n$ is a  connected component  of $\Gamma$  or there is  an edge $i\rightarrow n$. Then $c_n>c_i$ and  we set $\Delta=\tilde \Delta\cup\{n\}$ as the disjoint union. Therefore $\Delta\in M(\Gamma)$ and $M=M_{\Delta}$.

Now let us prove that our map is  injection. Suppose that $\Delta\ne\tilde\Delta$. Then  we can suppose that there exists $e\in E(\Delta),\,e\notin E(\tilde\Delta)$. Therefore $e=(i\rightarrow j)$ and $c_i<c_j$ and $i,j$ belong to the same connected component $\Delta_{\alpha}$ of the graph $\Delta$. In the same time they belong to the different connected components $\tilde\Delta_{\beta},\tilde\Delta_{\gamma}$ of the graph $\tilde\Delta$ ( indeed suppose that $i,j$ belong to same connected component $G$ in $\tilde\Delta$. Then there exists a path from minimal element in  $G$ to $i$ and a path to $j$. Adding to the first path $e$ we get  a cycle in $\Delta$. This is a contradiction).    If $k$ is the minimal element  in $\tilde\Delta_{\gamma}$ then $k\ge j$ (if $k<j$  then the path from the the minimal element of $\Delta_{\alpha}$ to $k$ and then to $j$ should be the same as  path from the minimal element of $\Delta_{\alpha}$  to $j$ calling at $i$. Therefore $k=j$. So we see that $x_{\Delta(j)}=x_{\Delta(i)}$ and $x_{\tilde\Delta(j)}=c_j>c_i\ge x_{\tilde\Delta(i)}$.
 \end{proof}
 According  to \cite{Bar} in order to describe the tangent cones we need to choose the inequalities (from the defining inequalities  of $D_{\Gamma}$)  which are active on the corresponding vertex.  In other words we need to choose those inequalities for which $M_{\Delta}$ is a solution of the corresponding equation.

\begin{corollary}\label{vertex1}  Let $\Delta=\cup_{\alpha}\Delta_{\alpha}$ be the decomposition into connected components and $i_{\alpha}$ be the minimal number in $\Delta_{\alpha}$.  Then  the  vertex $M_{\Delta}$ can be define by the characteristic function
$$
[M_{\Delta}]=\prod_{\alpha}[x_{i_{\alpha}}=c_{i_{\alpha}}]\prod_{(i\rightarrow j)\in E(\Delta)}[x_i= x_j].
$$

 The tangent cone  at the point $M_{\Delta}$ has the following characteristic function
$$
[tcone(M_{\Delta})]=\prod_{\alpha}[x_{i_{\alpha}}\le c_{i_{\alpha}}]\prod_{(i\rightarrow j)\in E(\Delta)}[x_i\le x_j]
$$
\end{corollary} 
\begin{proof}  
The polyhedron $D_{\Gamma}$ is given by the inequalities
 $$
l_i(N)=x_i\le c_i,\,i\in V(\Gamma),\quad l_{ij}(N)=x_i-x_j\le0\,\, \text{if}\,\,i\rightarrow j\in E(\Gamma).
$$
Suppose that  $l_i(M_{\Delta})=c_i$. Therefore $c_i=c_{i_{\alpha}}$. So we see that  $i=i_{\alpha}$. Now suppose that $l_{ij}(M_{\Delta})=0$ where $i\rightarrow j\in E(\Gamma)$.    We can suppose that $i\rightarrow j\notin E(\Delta)$. Therefore $c_{\Delta(i)}<c_{\Delta(j)}$ and we came to contradiction.
\end{proof}
We also need some different   (open) cone with the  same generating function for integer points. Let us set
$$
[tcone^{op}(M_{\Delta})]=\prod_{\alpha}[x_{i_{\alpha}}\le c_{i_{\alpha}}]\prod_{(i\rightarrow j)\in E(\Delta)}(1-[x_i\le x_j])
$$

We should note that $[tcone^{op}(M_{\Delta})]$ is not a polyhedron.
\begin{lemma}\label{vertex2}  The difference $[tcone(M_{\Delta})]-[tcone^{op}(M_{\Delta})]$ is a linear combination of characteristic function of rational polyhedra containing  a line.
\end{lemma}

\begin{proof}
We have 
$$
[tcone^{op}(M_{\Delta})]=(-1)^{E(\Delta)|}[tcone(M_{\Delta})]
+\prod_{\alpha}[x_{i_{\alpha}}\le c_{i_{\alpha}}]\sum_{p=1}^{q-1}(-1)^ps_{p},
$$
where    $s_{p}$ is the elementary  symmetric polynomial in characteristic functions of the edges of $\Delta$ and $q=|E(\Delta)|$.
So we see that we only need to prove that for any $p<q$ and any summand in $s_{p}$ the corresponding polyhedron contains a line. We can suppose that $\Delta$ is a connected graph. Let $\chi=\chi_{e_{i_1}}\dots\chi_{e_{i_{p}}}$ be one of the  summands in the decomposition of $s_{p}(e_1,\dots,e_q)$. Consider the graph $\Delta'$ with the same set of vertexes as $\Delta$ and the set of edges $e_{i_1},\dots,e_{i_p}$. Since $p<q$ There exist at least one connected component $\Delta'_1$  of $\Delta'$ which does not contain the minimal vertex. Therefore this connected component can be described by the inequalities
$$
x_i\le x_j,\,(i\rightarrow j)\in E(\Delta'_1)
$$
So we see that this connected  component contains a line $x_i=x_j=t,\,(i\rightarrow j)\in E(\Delta'_1)$. Lemma is proved
\end{proof}
 \begin{corollary} The cone  $tcone^{op}(M_{\Delta})$ can be described by the following inequalities: $x_i\le a,\,i=1,\dots,n$; if  $e=(i\rightarrow j)\in\Delta$ then $x_i>x_j$.
 
 \end{corollary}
 \begin{definition} Le  $\Gamma$ be a directed graph without multiple edges and cycles (if we ignore the orientation). Let us denote by $D_{\Gamma}(a)$ the following cone in $\Bbb R^r$
 $$
 [D_{\Gamma}(a)]=\prod_{i=1}^r[x_i\le a]\prod_{i\rightarrow j\in E(\Gamma)}[x_i\le x_j]
 $$
 Let us also denote by $S_{\Gamma}$ the set of the permutations of the vertexes of $\Gamma$ such that: if $i\rightarrow j$ is an edge of $\Gamma$ then $\sigma(i)<\sigma(j)$.
\end{definition}
\begin{lemma}\label{t6} The following  statements  hold true

$1)$
$$
D_{\Gamma}(a)=\bigcup_{\sigma\in S_{\Gamma}}D_{\Gamma}(a,\sigma)
$$
where 
$$
D_{\Gamma}(a,\sigma)=\{(x_1,\dots,x_n)\in\Bbb R^n\mid x_{\sigma^{-1}(1)}<\dots<x_{\sigma^{-1}(r)}\le a\}
$$

$2)$ If $\sigma,\tau\in S_{\Gamma}$ then $\tau^{-1}\sigma$ is a bijection between $D_{\Gamma}(a,\sigma)$ and $D_{\Gamma}(a,\tau)$.

$3)$ If $\Gamma^{op}$  is obtained from $\Gamma$  by changing all arrows on the opposite ones then $|S_{\Gamma}|=|S_{\Gamma^{op}}|$.

\end{lemma}
\begin{proof}  Let us prove  the first statement. Let  $\sigma\in S_{\Gamma}$ and  suppose that $M=(x_1,\dots,x_n)\in D_{\Gamma}(a,\sigma)$,  $i\rightarrow j\in E_{\Gamma}$. Then $\sigma(i)<\sigma(j)$ and 
$$
x_i=\,x_{\sigma^{-1}(\sigma(i))}<x_{\sigma^{-1}(\sigma(j))}=x_j
$$ Therefore  $M\in D_{\Gamma}(a)$. 
Now let $M\in D_{\Gamma}(a)$. There exists a unique $\sigma$ such that $x_{\sigma^{-1}(1)}<\dots<x_{\sigma^{-1}(n)}$ and we only need to prove that $\sigma\in S_{\Gamma}$. Let $[i,j]\in E_{\Gamma}$ then since $M\in D_{\Gamma}(a)$ we have 
 $$
 x_{\sigma^{-1}(\sigma(i))}=x_i<x_{j}=x_{\sigma^{-1}(\sigma(j))}.
 $$
 Therefore $\sigma(i)<\sigma(j)$. 
 
 Let us prove the second statement.  If $M=(x_1,\dots,x_r)\in D(a,\sigma)$ and $i<j$ then (since $\tau\in S_{\Gamma}$)  $\tau(i)<\tau(j)$. Therefore $x_{\sigma^{-1}\tau(i)}<x_{\sigma^{-1}\tau(j)}$ and $\sigma^{-1}\tau(M)\in D_{\Gamma}(a, \tau)$.
 
Let  $w(i)=r-i+1$   be a permutation of the  vertexes. Then it is easy to see that $w$ is a bijection of $S_{\Gamma}$ into $S_{\Gamma^{op}}$ and the third statement follows. 
\end{proof}
\begin{definition}
Let $\Gamma$ be a  directed graph without multiple edges  and cycles (if we ignore the orientation of  $\Gamma$) with $r$ vertexes enumerated by integers  $1,\dots,r$. Let also  $c_1< c_2<\dots<c_r$ be a sequence of integers such that if $i\rightarrow j$ is a directed edge then $c_i< c_j$.  Let us also define $c_{\Gamma(i)}$ to be  equal to  the $\min\{c_k\}$  where $c_k$ run over  connected component containing $c_i$. Let us define  a Laurent  polynomial 
$$
\theta(\Gamma,t_1,\dots,t_r)=\frac{1}{r!}\sum_{\Delta\in M(\Gamma)}(-1)^{|E(\Delta|}|S_{\Delta}|\prod_{i=1}^rt_i^{ c_{\Delta(i)}-c_i}
$$ 
\end{definition} 
The following Proposition shows  that function $\theta$ is multiplicative.
\begin{proposition}\label{mult} If $\Gamma=\Gamma_1\cup\Gamma_2$ is a disjoint union of two graphs, then 
$$
\theta(\Gamma_1,t)\theta(\Gamma_2,s)=\theta(\Gamma,t,s)
$$
\end{proposition}
\begin{proof}  We have 
$$
\theta(\Gamma_1,t)\theta(\Gamma_2,s)=\frac{1}{(r_1)!}\frac{1}{(r_2)!}\sum_{\Delta_1\in M(\Gamma_1),\Delta_2\in M(\Gamma_2)}(-1)^{|E(\Delta_1|+|E(\Delta_2)|}
$$
$$
\times|S_{\Delta_1}||S_{\Delta_2}|\prod_{i\in V(\Gamma_1)}t_i^{ c_{\Delta(i)}-c_i}\prod_{j\in V(\Gamma_2)}s_j^{ c_{\Delta(j)}-c_j}
$$
It is easy to see that   any $\Delta\in M(\Gamma)$ can be uniquely represented in the form $\Delta=\Delta_1\cup\Delta_2$ where $(\Delta_1,\Delta_2)\in M(\Gamma_1)\times M(\Gamma_2)$.  Further  we see that
$$
|S_{\Delta_1\cup\Delta_2}|=|S_{\Delta_1}||S_{\Delta_2}|\frac{(n_1+n_2)!}{(n_1)! (n_2)!}
$$
and
$$
|E(\Delta_1)|+|E(\Delta_2)|=|E(\Delta)|
$$
Proposition is proved.
\end{proof}

\section{Character formula}

\begin{definition} Let $f$ be a diagram such that  
$$
f^{-1}(>,\times)=\{a_1>\dots>a_m\},\quad f^{-1}(<,\times)=\{b_1<\dots<b_n\},\,f^{-1}(\times)=r
$$
Let us define the following evaluation homomorphism 
$$
ev_f : \Bbb R[t_1,\dots,t_N]  \rightarrow  \Bbb R[e^{\pm\varepsilon_1},\dots,e^{\pm\varepsilon_m}, e^{\pm\delta_1},\dots, e^{\pm\delta_n}],\,\,N=m+n-r
$$
by the following rule:  
$$
ev_f(t_k)=\begin{cases} e^{\varepsilon_i},\, \text{if}\,\,c_k=a_i,\, f(a_i)=\,\,>\\
e^{-\delta_j},\,\,\text{if}\,\,c_k=b_j ,\,f(b_j)=\,\,<\\
-e^{\varepsilon_i-\delta_j}\,\,\text{if}\,\,c_k=a_i=b_j ,\,f(a_i)= f(b_j)=\times
\end{cases}
$$
\end{definition}

\begin{lemma}\label{t7} The following formula holds true
$$
J(ev_f(G(tcone(M_{\Delta})))=\frac{|S_{\Delta}|}{|V(\Delta)|!}J\left(ev_f\left(\prod_{i\not\in f^{-1}(\times)}t_i^{c_i}\prod_{i\in f^{-1}(\times)}\frac{ t_i^{c_{\Delta(i)}}}{(1-t_i^{-1})}\right)\right)
$$
\end{lemma} 
\begin{proof} By Proposition \ref{mult} we can suppose that $\Delta$ is a connected graph and  set $a=c_{\Delta_{i}}$. Therefore we can apply   Lemma \ref{t6} to $tcone^{op}(M_{\Delta})=D_{\Delta^{op}}(a)$.   According to this Lemma
$$
J(ev_fG(D_{\Delta^{op}}(a)))=|S_{\Delta^{op}(a)}|J(ev_fG(D_{\Delta^{op}}(a)(id)))
$$
and it is easy to see that
$$
J(ev_fG(D_{\Delta^{op}}(a)(id)))=\frac{1}{|V(\Delta)|!}J\left(ev_f\left(\prod_{i\not\in f^{-1}(\times)}t_i^{c_i}\prod_{i\in f^{-1}(\times)}\frac{ t_i^{c_{\Delta(i)}}}{(1-t_i^{-1})}\right)\right)
$$
\end{proof}

Now we are going to give a formula for characters  of irreducible finite dimensional representations  of  the   superalgebra Lie $\frak{gl}(m,n)$. Let  $\chi$ be an integer dominant weight  then we can define a pair of the set $(A,B)$ and the diagram $f$. We also define the $S_{\chi}$ - a maximal  $\chi+\rho$  isotropic set. The set $S_{\chi}$ is in  a bijection with the set of $f^{-1}(\times)$ by the following rule: if $f(c)=\times$ and $c=a_i=b_j$ then $\varepsilon_i-\delta_j\in S_{\chi}$. So we can order the set 
$S_{\chi}=\{\alpha_1<\dots<\alpha_r\}$ using the natural order on  $f^{-1}(\times)$ and we also have the corresponding graph $\Gamma_f$. 
 Let  also
$$
D=\frac{\prod_{\alpha\in R_0^+}\left(e^{\alpha/2}-e^{-\alpha/2}\right)}{\prod_{\alpha\in R_1^+}(e^{\alpha/2}+e^{-\alpha/2})}.
$$
\begin{thm}\label{main} The following equality holds true

\begin{equation}\label{F}
D ch\,L(\chi)=\sum_{w\in W_0}\varepsilon(w)w\left(\frac{e^{\chi+\rho}\theta(\Gamma_f,-e^{\alpha_1},\dots,-e^{\alpha_r})}{\prod_{\alpha\in S_{\chi}}(1+e^{-\alpha})}\right)
\end{equation}
\end{thm}
\begin{proof} By Collorary \ref{t2} we have
$$
D ch\,L(f)=(-1)^{S(f^{-1}(\times))}\sum_{x\in\Bbb Z^N\cap D_f}(-1)^{S(g_x^{-1}(\times))}J(e^{\pi_f(x)})
$$
$$
(-1)^{S(f^{-1}(\times))}\sum_{x\in\Bbb Z^N\cap D_f}J(ev_f(t_1^{x_1}\dots t_N^{x_N}))=
$$
$$
(-1)^{S(f^{-1}(\times))}J(ev_f(G(D_f)))
$$
So we need to calculate the function $J(G(D_f))$. By Brion's theorem we have 
$$
G(D_f)=\sum_{M\in V} G(tcone(M)).
$$
Therefore we need to calculate  $J(ev_f(G(tcone(M))))$ for any vertex $M$ of $D_f$. By Lemma \ref{vertex}  any vertex can be described by means  of graph $\Delta$. So by Lemma \ref{t7} we have 
$$
D ch \,L(f)=
$$
$$
(-1)^{S(f^{-1}(\times))}J\left(ev_f\left(\sum_{\Delta}(-1)^{|E(\Delta)|}\frac{|S_{\Delta}|}{|V(\Delta)|!}\frac{\prod t_i^{c_{\Delta(i)}}}{\prod(1-t_i^{-1})}\right)ev_f(\prod_{f(c_i)\ne\times} t_i^{c_i})\right)=
$$
$$
(-1)^{S(f^{-1}(\times))}J\left(ev_f\left(\prod_{} t_i^{c_i}
\sum_{\Delta\in M(\Gamma)}(-1)^{|E(\Delta)|}\frac{|S_{\Delta}|}{|V(\Delta)|!}\frac{\prod t_i^{c_{\Delta(i)}-c_i}}{\prod(1-t_i^{-1})}\right)\right)
$$
But
$$
ev_f\left(\prod_{i\in f^{-1}(\times)}(1-t_i^{-1})\right)=\prod_{\alpha\in S_{\lambda}}(1+e^{-\alpha})
$$
and
$$
\frac{1}{|V(\Gamma)|!}\sum_{\Delta\in M(\Gamma)}(-1)^{|E(\Delta)|}|S_{\Delta}|\prod t_i^{c_{\Delta(i)}-c_i}=\theta(\Gamma,t_{i_1},\dots, t_{i_r})
$$
where $i_1<i_2\dots<i_r=f^{-1}(\times)$.
\end{proof}
\begin{remark} From  Proposition \ref{mult} we see that we only need to write down the character formula only  for connected  graphs.
\end{remark}
\begin{example}
 We will consider the most atypical case of $\frak{gl}(3,3)$. Let   us consider the following weight 
 $$
 \chi=(c+3)\varepsilon_1+(c+2)\varepsilon_2+(c+2)\varepsilon_3-(c+5)\delta_1-(c+3)\delta_2-(c+3)\delta_3
 $$
 $$
 \chi+\rho=(c+3)(\varepsilon_1-\delta_3)+(c+1)(\varepsilon_2-\delta_2)+c(\varepsilon_3-\delta_1)
 $$
 $$
 S_{\chi}=\{\alpha_1<\alpha_2<\alpha_3\}.
 $$
 The corresponding weight diagram is   $f^{-1}(\times)=\{c,c+1,c+3\}, f^{-1}(<)=f^{-1}(>)=\emptyset$.  The corresponding  cap diagram is  the following 

\vskip0.5cm
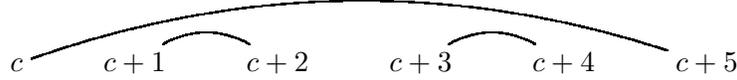
$$
\xymatrix{c\ar@/^2pc/@{}[rrrrr]&c+1\ar@/^1pc/@{}[r] & c+2 &c+3\ar@/^1pc/@{}[r] & c+4& c+5}
$$

 In this case  $\Gamma_f$ is the following graph
$$
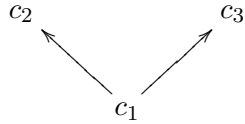
\xymatrix{c_2&&c_3\\
 & c_1\ar[ul]\ar[ur]&} 
$$
where  $c_1=c,c_2=c+1,c_3=c+3$. The corresponding polyhedron  is
$$
D_f=\{x_1\le c_1,\,x_2\le c_2,\,x_3\le c_3,\,x_1\le x_2,\,x_1\le x_3\}.
$$
The polyhedron $D_f$ has four vertexes 
$$
M_1=(c_1,c_1,c_1),\,M_2=(c_1,c_2,c_1),\,M_3=(c_1,c_1,c_3),\,M_4=(c_1,c_2,c_3),\,
$$
and these vertexes correspond  to the  four possibilities for graph $\Delta$ 
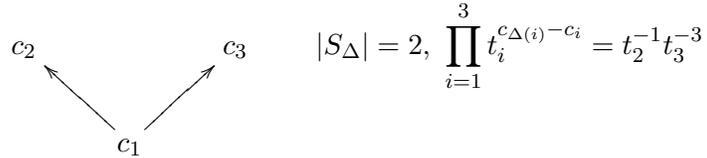
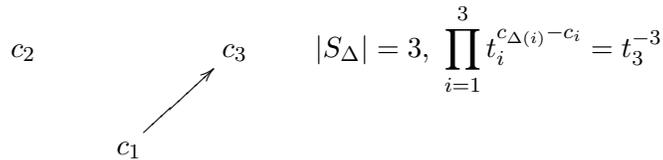
$$
\xymatrix{c_2&&c_3\\
 &c_1\ar[ul]\ar[ur]&}
 \quad\quad |S_{\Delta}|=2, \,\, \prod_{i=1}^3t_i^{ c_{\Delta(i)}-c_i}=t^{-1}_2t_3^{-3}
 $$
$$
 \xymatrix{c_2&&c_3\\
 & c_1\ar[ur]&} 
 \quad\quad |S_{\Delta}|=3, \,\, \prod_{i=1}^3t_i^{ c_{\Delta(i)}-c_i}=t_3^{-3}
 $$
 
 $$
 \xymatrix{c_2&&c_3\\
 & c_1\ar[ul]&}\quad\quad 
  |S_{\Delta}|=3, \,\, \prod_{i=1}^3t_i^{ c_{\Delta(i)}-c_i}=t_2^{-1}
 $$
 
 $$
 \xymatrix{c_2&&c_3\\
 & c_1&} 
 \quad\quad  |S_{\Delta}|=6, \,\, \prod_{i=1}^3t_i^{ c_{\Delta(i)}-c_i}=1
$$
Therefore in this case we have 
$$
 \theta(\Gamma,t_1,t_2,t_3)= 1-\frac12t_2^{-1}-\frac12t_3^{-3}+\frac13 t_2^{-1}t_3^{-3}
 $$
So  we have the following equality
$$
D\,chL(\chi)=J\left(e^{\chi+\rho}\frac{1+\frac12e^{-\alpha_2}+\frac12e^{-3\alpha_3}+\frac13e^{-\alpha_2-3\alpha_3}}{(1+e^{-\alpha_1})(1+e^{-\alpha_2})(1+e^{-\alpha_3})}\right)
$$
\end{example}

Actually more general result can be proved by the same way.

Let $f$ be a diagram. Then we can represent $f^{-1}(\times)=\cup[[c_i,d_i]]$ as the disjoin union of the of integer segments and define $\tilde c_i$ as the maximum $c_j$ where $c_j$ run over the segment containing $c_i$.
Let $\Gamma_f$  be  the same graph as before and  edge  $i\rightarrow j$  is called special  if $\tilde c_i<\tilde c_j$. Let us also define $c_{\Gamma(i)}$ to be  equal to  the $\min\{c_k\}$  where $c_k$ run over  connected component containing $c_i$.  We denote by $M(\Gamma)$ the set of subgraphs of $\Gamma$ with the same set of vertexes and the same set of non special  edges.
 Let us  also denote by $\Gamma_0$ the subgraph of $\Gamma$ consisting of vertexes such that $\tilde c_i=\tilde c_{\Gamma(i)}$ and the edges of $\Gamma_0$ are the same as in $\Gamma$. It is easy to see that if $\Delta\in M(\Gamma)$ then $\Delta\supset \Gamma_0$.
Let us  define the following  Laurent  polynomial 
$$
\tilde \theta(\Gamma,t_1,\dots,t_n)=\frac{1}{n!}\sum_{\Delta\in M(\Gamma)}(-1)^{|E(\Delta\setminus\Delta_0)|}|S_{\Delta^*}|\prod_{i=1}^nt_i^{ c_{\Delta(i)}-c_i}
$$
 where  $\Delta^*$ is  the graph which can be obtained from $\Delta$ by  replacing all edges  in $\Delta\setminus\Delta_0$ by the opposite ones.

\begin{thm}
 The following equality holds true

\begin{equation}\label{F}
D\,chL(\chi)=(-1)^{\nu}J\left(\frac{e^{\chi+\rho+\gamma}\tilde\theta(\Gamma_f,-e^{\alpha_1},\dots,-e^{\alpha_r})}{\prod_{\alpha\in S_{\chi}}(1+e^{-\alpha})}\right)
\end{equation}
where $ \nu=\sum_{i}(\tilde c_i-c_i)$ and $\gamma=\sum_{i}(\tilde c_i-c_i)\alpha_i$.
\end{thm}
\begin{example}
Let us suppose that $c_1=c,\,c_2=c+1, c_3=c+3$.  Then we have $\tilde c_1=\tilde c_2=c+1,\,\tilde c_3=c+3$. Therefore   we have two possibilities for graph $\Delta$ and the corresponding summands
$$
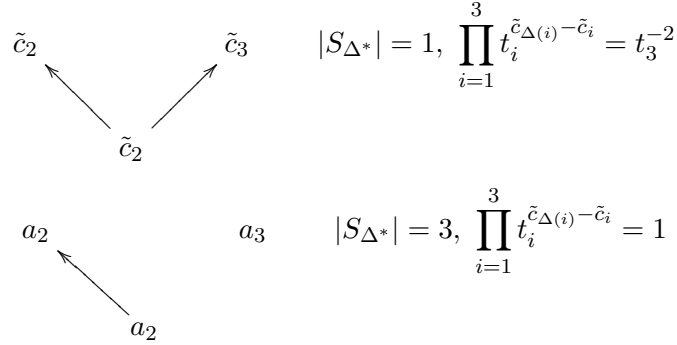
\xymatrix{\tilde c_2&&\tilde c_3\\
 &\tilde c_2\ar[ul]\ar[ur]&}
 \quad\quad |S_{\Delta^*}|=1, \,\, \prod_{i=1}^3t_i^{ \tilde c_{\Delta(i)}-\tilde c_i}=t_3^{-2}
 $$
 $$
 \xymatrix{a_2&&a_3\\
 & a_2\ar[ul]&}\quad\quad 
 |S_{\Delta^*}|=3, \,\, \prod_{i=1}^3t_i^{ \tilde c_{\Delta(i)}-\tilde c_i}=1
 $$
 So we see that  $S=1$ and $\gamma=\alpha_1$.
Therefore
$$
D\,chL(\chi)=-J\left(e^{\chi+\rho+\alpha_1}\frac{\frac12-\frac12e^{-2\alpha_3}}{(1+e^{-\alpha_1})(1+e^{-\alpha_2})(1+e^{-\alpha_3})}\right)
$$
\end{example}
\begin{example} (See \cite{CHS}) Suppose that diagram $f$ is PDC. This means that the corresponding graph $\Gamma=\Gamma_f$ is a disjoint union of segments. It is easy to check that in this case $\Gamma_0=\Gamma$. Therefore $\tilde\theta(\Gamma,t_1,\dots,t_r)=\frac{|S_{\Gamma}|}{|V(\Gamma)|!}$ and we have the following formula for irreducible character
$$
D\,ch L(f)=(-1)^{\nu}\frac{|S_{\Gamma}|}{|V(\Gamma)|!} J\left(\frac{e^{\chi+\rho+\gamma}}{\prod_{\alpha\in S_{\chi}}(1+e^{-\alpha})}\right).
$$
But $|S_{\Gamma}|=\frac{(n_1+\dots+n_r)!}{(n_1)!\dots (n_r)!}$. So this formula coincides with  the formula for $PDC$ modules in \cite{CHS}.
\end{example}

\end{document}